\documentclass[a4paper,10pt]{article}
\usepackage{amssymb}
\usepackage{amstext}
\usepackage{amsmath}
\usepackage{amscd}
\usepackage{amsthm}
\usepackage{amsfonts}
\usepackage{enumerate}
\usepackage{latexsym}
\usepackage{mathrsfs}
\usepackage{mathtools}

\makeatletter
 
 \@addtoreset{equation}{section}
\makeatother

\newtheorem{theorem}{Theorem}[section]
\newtheorem{lemma}[theorem]{Lemma}
\newtheorem{proposition}[theorem]{Proposition}
\newtheorem{remark}[theorem]{Remark}

\newtheorem{definition}[theorem]{Definition}

\DeclareMathOperator{\wn}{Wind}

\DeclareMathOperator{\Trace}{Tr}
\DeclareMathOperator{\trace}{tr}

\DeclareMathOperator{\Ln}{Ln}

\DeclareMathOperator{\Index}{Index}

\def\C{\mathbb C}
\def\N{\mathbb N}
\def\S{\mathbb S}
\def\R{\mathbb R}
\def\Z{\mathbb Z}

\def\B{\mathcal B}
\def\d{\mathrm d}
\def\D{\mathcal D}
\def\e{\mathrm e}
\def\E{\mathcal E}
\def\F{\mathscr F}
\def\H{\mathcal H}
\def\I{\mathcal I}
\def\J{\mathcal J}
\def\K{\mathcal K}
\def\SS{\mathcal S}

\def\U{\mathcal U}
\def\W{\mathcal W}
\def\one{I}

\def\Dom{\mathcal D}
\def\p{{\rm p}}
\def\i{{\rm i}}
\def\r{{\rm r}}
\def\Ka{\mathcal K}
\def\Ia{\mathcal I}
\def\Ha{\mathcal H}
\def\Ja{\mathcal J}
\def\Ya{\mathcal Y}

\newcommand\cF{{\mathscr F}}

\begin{document}

\title{Topological Levinson's theorem for inverse square potentials:
complex, infinite, but not exceptional}

\author{H. Inoue, S. Richard\footnote{Supported by the grant
\emph{Topological invariants through scattering theory and noncommutative geometry} from Nagoya University,
and on leave of absence from Univ.~Lyon, Universit\'e Claude Bernard Lyon
1, CNRS UMR 5208, Institut Camille Jordan, 43 blvd.~du 11 novembre 1918, F-69622
Villeurbanne cedex, France}}

\date{\small}
\maketitle \vspace{-1cm}

\begin{quote}
\emph{
\begin{itemize}
\item[] Graduate school of mathematics, Nagoya University,
Chikusa-ku, Nagoya 464-8602, Japan
\item[] \emph{E-mails:} inoue.hideki124@gmail.com, richard@math.nagoya-u.ac.jp
\end{itemize}
}
\end{quote}

\begin{abstract}
In this review paper we carry on our investigations on Schr\"odinger operators with inverse square potentials on the half-line.
Depending on several parameters, such operators possess either a finite number
of complex eigenvalues, or an infinite one, but also some spectral singularities embedded in the continuous spectrum (exceptional situations). The spectral and the scattering theory for these operators is recalled, and new results for the exceptional cases are provided.
Some index theorems in scattering theory are also developed, and explanations why these results can not be extended to the exceptional cases are provided.
\end{abstract}

\smallskip



\section{Introduction}

Levinson's theorem is a relation between the number of bound states of a quantum mechanical system
and an expression related to the scattering part of that system.
It was originally established by N. Levinson in \cite{Lev} for Schr\"odinger operators with a spherically symmetric potential,
and has then been developed by numerous researchers on a purely analytical basis.
About 10 years ago, it has been shown that this relation can be interpreted as an index theorem in scattering theory, and the results of these investigations have been
summarized in the review paper \cite{R}.
More recently, a scattering system involving several parameters has been exhibited in \cite{DR} and this system has been at the root of several extensions of Levinson's theorem: In \cite{NPR} it has been shown that complex eigenvalues can also be counted, and in \cite{IR} a first attempt for dealing with an infinite number of eigenvalues has been introduced. However, some of the operators exhibited in \cite{DR} were not used for these extensions, and our aim in the present paper is to complete the investigations
for the entire family.

Before entering into the details of our investigations, let us immediately mention
that part of our aim has been unsuccessful. Indeed, for a reduced family of operators, we end up with wave operators which are either unbounded or not Fredholm. In such a situation, computing their Fredholm index is either much more involved or simply not possible. Nevertheless, we provide an exhaustive picture of the situation and describe the limitations of our approach. We hope that our presentation will motivate further investigations for the trickiest cases.

Let us now be more precise on the model and on the results, see also Section \ref{sec_model} for more details on the model.
The initial system consists in a family of Schr\"odinger operators of the form
$-\partial_r^2+\big(m^2-\frac{1}{4}\big)\frac{1}{r^2}$
on the half-line $\R_+$.
The parameter $m\in \C$ with $\Re(m)>-1$ is used for describing the coupling constant for the potential.
For $m\neq 0$ an additional parameter $\kappa\in \C$ is used
for defining the boundary condition at $r=0$,
while for $m=0$ another family of operators indexed by a boundary parameter $\nu$ is defined.
The study of the corresponding families of closed operators $H_{m,\kappa}$ and $H_0^\nu$ in $L^2(\R_+)$ has been initiated and extensively performed in \cite{DR}.

Among all operators $H_{m,\kappa}$ and $H_0^\nu$ only a few are self-adjoint.
They are exhibited in Lemma \ref{lemma_SA}. In the large complementary family, some
pairs of parameters $(m,\kappa)$ and some parameters $\nu$ are called \emph{exceptional} if they satisfy a prescribed condition provided in Definition
\ref{def_exceptional}. As shown in Remark \ref{rem_sin} the corresponding operators
$H_{m,\kappa}$ or $H_0^\nu$ possess \emph{spectral singularities} in the continuous spectrum. Around these singularities the spectral and the scattering properties of these operators are less obvious, and for that reason these operators were not considered in \cite{DR}. Here, we shall consider all the operators, and
provide as much information as possible even in the exceptional situations.

The spectral theory of the operators $H_{m,\kappa}$ and $H_0^\nu$ is provided in
Section \ref{sec_spect}. The number of eigenvalues of these operators can be finite or infinite, depending on the parameters. For the exceptional operators, it is shown in particular that even though there is no eigenvalue embedded in the continuous spectrum, it is possible to construct a family of operators of the same type (but non-exceptional) having complex eigenvalues converging to a prescribed value in $\R_+$, see Lemma \ref{lem_acumul}.
These convergences take place either in $\C_+$ or in $\C_-$ depending on the choice of the initial exceptional parameters.

The next spectral result corresponds to a limiting absorption principle. For non-exceptional operators this limiting process takes place from below and from above the real axis in $\C$, but in the exceptional cases some restrictions appear. More precisely, at the spectral singularity the limiting absorption principle holds only on one side of the real axis, the side free of possible accumulations of complex eigenvalues. These results are gathered in Propositions \ref{prop_1} and \ref{prop_2}.
Let us also note that even if most of the exceptional operators have only one spectral singularity, some have two spectral singularities (with corresponding limiting absorption principles in two different half-planes) and some have an infinite number of spectral singularities, converging both to $0$ and to $+\infty$.

The scattering theory for the pairs of operators $(H_{m,\kappa},H_{\rm D})$ or $(H_0^\nu,H_{\rm D})$ is studied in Section \ref{sec_scat}. The reference operator $H_{\rm D}$ corresponds to the Dirichlet Laplacian on $\R_+$ (and is equal to $H_{\frac{1}{2},0}$).
Following the approach of \cite{DR} we start by constructing the generalized Hankel transformations $\F_{m,\kappa}^\mp$ and $\F_0^{\nu \mp}$, and define the wave operators in terms of these transformations. Various representations are provided for these operators, but here again a special attention has to be paid to the exceptional cases. Indeed, for them either one or sometimes both wave operators are not bounded.
A list of all unbounded wave operators is provided at the end of Section \ref{sec_scat}.

In the last section we provide some index theorems in scattering theory. This part contains new information but the framework corresponds to the one which already appeared in \cite{NPR} and to part of the one used in \cite{IR}.
The first step consists in providing a representation of the wave operators in the usual setting of pseudo-differential operators. Since this new representation is implemented by a unitary transformation, the bounded wave operators remain bounded, and the unbounded ones remain unbounded ! However, this representation is convenient for the
introduction of some $C^*$-algebras containing pseudo-differential operators
of order $0$ and with coefficients which are either asymptotically constant or periodic. These algebras contain all bounded wave operators, as stated in Proposition
\ref{prop_enumeration}.

Once in this $C^*$-algebraic framework, the way for index theorems is already paved and rather well understood. Indeed, by looking at some ideals in these algebras and by considering the quotient algebras, one ends up automatically with an index map which corresponds to our topological version of Levinson's theorem. For the model under consideration and depending on the parameters, one obtains either an index theorem for Fredholm operator or a so-called \emph{Atiyah's $L^2$-index theorem} \cite{Ati}.
These results are presented in Theorem \ref{thm_1} and \ref{thm_2}.
Note that the Fredholm case has already been considered for several models in \cite{R}, and beside the usual contribution due to the scattering operator, two additional contributions are possible. In the original representation they correspond to corrections at $0$-energy and at energy equal to $+\infty$.
Note also that in the present setting we are dealing with arbitrary complex eigenvalues while in reference \cite{R} only real eigenvalues were considered.
On the other hand, when the number of bound states is infinite, no such correction appears, and the new Levinson's theorem corresponds to an equality between the winding number computed over one period for the scattering operator, and a suitable trace in the Floquet-Bloch representation of the projection on the bound states of $H_{m,\kappa}$.
This situation coincides with a special instance of the results obtained in the seminal paper \cite{Ati} where
an index theorem is provided for elliptic operators on a non-compact manifold which are invariant under the action
of a discrete group.
The decomposition with respect to the group corresponds in our setting to the Floquet-Bloch decomposition.

Unfortunately, in Section \ref{sec_index} about index theorems the exceptional cases are no more considered. Indeed, as already mentioned some of the corresponding wave operators are unbounded, and therefore can not easily be associated to any $C^*$-algebra. For the remaining wave operators still belonging to some $C^*$-algebras, their principal symbol are not boundedly invertible. As a consequence, these operators are not Fredholm, and their analytical index can not be defined. It is quite unfortunate that the presence of a spectral singularity prevented us from defining any index theorem for the corresponding wave operators. Note that a related result about the non-completeness of the wave operators in the presence of spectral singularity has also been recently obtained in \cite{FF}.
Some relations between spectral singularities and scattering theory have also been exhibited in \cite{M,S}, see also references therein.
We hope that future investigations will provide new insights about these exceptional situations in our algebraic framework.

\section{The model}\label{sec_model}

In this section we introduce the model used for our investigations.
This material is borrowed from \cite{DR} to which we refer for more explanations and for the proofs.
Note that \cite{IR,NPR} also contain partial information of this model.

For any $m \in \C$ we consider the differential expression
\begin{equation*}
L_{m^2}:=-\partial_r^2+\Bigl(m^2-\frac{1}{4}\Bigr)\frac{1}{r^2}
\end{equation*}
acting on distributions on $\R_+$.
The maximal operator associated with $L_{m^2}$ in $L^2(\R_+)$ is defined by
$\D(L_{m^2}^{\max})=\{f\in L^2(\R_+)\mid L_{m^2}f\in L^2(\R_+)\}$, and the minimal operator
$L_{m^2}^{\min}$ is defined  as the closure of the restriction of $L_{m^2}$ to $C_{c}^\infty(\R_+)$,
where $C_{c}^\infty(\R_+)$ denotes the set of compactly supported smooth functions on $\R_+$.
Then, the equality $(L_{m^2}^{\min})^*=L_{\bar{m}^2}^{\max}$ holds for any $m\in\C$,
and $L_{m^2}^{\min}=L_{m^2}^{\max}$ if $|\Re(m)|\geq 1$ while
$L_{m^2}^{\min}\subsetneq L_{m^2}^{\max}$ if $|\Re(m)|<1$.
In the latter situation $\D(L_{m^2}^{\min})$ is a closed subspace of codimension $2$ of $\D(L_{m^2}^{\max})$,
and for any $f\in\D(L_{m^2}^{\max})$ there exist $a,b\in\C$ such that
\begin{align*}
f(r)  - ar^{1/2-m} - br^{1/2+m}&\in\Dom(L_{m^2}^{\min})\hbox{
around }0  \qquad \hbox{ if } m \neq 0, \\
f(r)-ar^{1/2}\ln(r)- br^{1/2}&\in\Dom(L_{0}^{\min})\hbox{
around }0.
\end{align*}
Here, the expression $g(r)\in\D(L_{m^2}^{\min})$ \emph{around} $0$ means that there exists
$\zeta\in\ C_c^{\infty}\big([0,\infty)\big)$ with $\zeta=1$ around $0$ such that $g\zeta\in\D(L_{m^2}^{\min})$.
In addition, the behavior of any function $g\in \Dom(L_{m^2}^{\min})$ is known, namely $g\in \H^1_0(\R_+)$ and as $r\to 0$~:
\begin{align*}
g(r) & = o\big(r^{3/2}\big)\ \hbox{ and }\ g'(r)= o\big(r^{1/2}\big)  \quad \hbox{ if } m \neq 0, \\
g(r) & = o\big(r^{3/2}\ln(r)\big)\ \hbox{ and }\ g'(r)= o\big(r^{1/2}\ln(r)\big)  \quad \hbox{ if } m = 0.
\end{align*}

\begin{remark}
It is worth mentioning that for $m\neq 0$ the functions $r \mapsto r^{\frac{1}{2}\pm m}$ are the two linearly independent solutions
of the ordinary differential equation $L_{m^2}u=0$, and that they are square integrable near $0$ if $|\Re(m)|<1$.
Similarly, the functions $r\mapsto r^{\frac{1}{2}}$ and  $r\mapsto r^{\frac{1}{2}}\ln(r)$
are the two linearly independent solutions of the ordinary differential equation $L_{0}u=0$.
\end{remark}

Based on the above observations we construct various closed extensions of the operator $L_{m^2}^{\min}$.
For simplicity we restrict our attention to $m\in \C$ with $|\Re(m)|<1$. These extensions
are parameterized by a boundary condition at $0$,
namely for any $\kappa\in \C\cup \{\infty\}$ we define a
family of closed operators $H_{m,\kappa}$~:
\begin{align*}
\Dom(H_{m,\kappa}) & = \big\{f\in \Dom(L_{m^2}^{\max})\mid
\hbox{ for some }  c \in \C,\\
&\qquad f(r)- c
\big(\kappa r^{1/2-m} +
r^{1/2+m}\big)\in\Dom(L_{m^2}^{\min})\hbox{ around }
0\big\},\qquad\kappa\neq\infty;  \\
\Dom(H_{m,\infty}) & = \big\{f\in \Dom(L_{m^2}^{\max})\mid
\hbox{ for some }  c \in \C,\\
&\qquad f(r)- c
r^{1/2-m}\in\Dom(L_{m^2}^{\min})\hbox{ around } 0\big\}.
\end{align*}

For $m=0$, we introduce an additional family of closed operators $H_0^\nu$ with
$\nu\in\C\cup \{\infty\}$~:
\begin{align*}
\Dom(H_0^\nu) & = \big\{f\in \Dom(L_{0}^{\max})\mid
\hbox{ for some }  c \in \C,\\
&\qquad f(r)- c
\big( r^{1/2}\ln(r) + \nu r^{1/2}\big)\in\Dom(L_0^{\min})\hbox{ around } 0\big\},\qquad\nu\neq\infty; \\
\Dom(H_0^\infty) & = \big\{f\in \Dom(L_0^{\max})\mid
\hbox{ for some }  c \in \C,\\
&\qquad f(r)- c
r^{1/2}\in\Dom(L_0^{\min})\hbox{ around }
0\big\}.
\end{align*}

Let us directly mention a few simple properties of these operators.
For any $|\Re(m)|<1$ and any $\kappa\in \C\cup \{\infty\}$,
the equality $H_{m,\kappa}=H_{-m,\kappa^{-1}}$ holds.
For that reason, the case $\kappa=\infty$ will be disregarded in the following.
In addition, the operator $H_{0,\kappa}$ does not depend on $\kappa$,
and all these operators coincide with $H_0^\infty$ (which
has already been fully investigated in \cite{BDG}).
For that reason, all results about the case $m=0$
will be formulated in terms of the family $H_0^\nu$ for $\nu \in \C$.
It has also been proved in \cite[Prop.~2.3]{DR} that
for any $m\in \C$ with $|\Re(m)|<1$ and for any $\kappa, \nu \in \C$ one has
\begin{equation*}\label{Eq_adjoint}
(H_{m,\kappa})^*=H_{\bar m,\bar\kappa}\qquad \hbox{ and }\qquad
(H_0^\nu)^* = H_0^{\bar \nu}.
\end{equation*}
Based on this, the self-adjoint elements can easily be identified
in the two families of operators. Indeed, one has:

\begin{lemma}\label{lemma_SA}
\begin{enumerate}
\item[(i)] The operator $H_{m,\kappa}$ is self-adjoint for $m\in(-1,1)$ and $\kappa\in\R$,
and for $m\in i\R$ and $|\kappa|=1$.
\item[(ii)] The operator $H_0^\nu$ is self-adjoint for $\nu\in \R$.
\end{enumerate}
\end{lemma}

Let us finally note that the operator $H_{\frac{1}{2},0}$ corresponds to the \emph{Dirichlet Laplacian} on $\R_+$,
which will be denoted by $H_{\rm D}$.
Later on, this operator will play the role of a comparison operator.

\section{Spectral theory}\label{sec_spect}

In this section  we start by introducing the definition of an exceptional pair $(m,\kappa)$
or of an exceptional parameter $\nu$.
We then show how these exceptional situations show off in spectral theory,
by recalling a few spectral result obtained in \cite{DR}
and by making some of them slightly more accurate.

In the sequel, we shall use the notations $z\mapsto \ln(z)$ for the principal value of the logarithm whose imaginary part lies in the interval $(-\pi, \pi]$.
On the other hand, $\Ln(z)$ will denote the multivalued logarithm. This means that $\Ln(z)=\{\ln(z)+2\pi i \Z\}$, or equivalently if $w$
satisfies $\e^{w}=z$, then $\Ln(z)$ is the set $\{w+2\pi i\Z\}$.
We also introduce for $m\in \C^*$ with $|\Re(m)|<1$ and for $\kappa\in \C$ the new parameter
\begin{equation*}\label{short}
\varsigma\equiv \varsigma(m,\kappa):=
\kappa\frac{\Gamma(-m)}{\Gamma(m)}
\end{equation*}
where $\Gamma$ denotes the usual Gamma function.

\begin{definition}\label{def_exceptional}
A pair $(m,\kappa)$ in $\C^*\times \C^*$ with $|\Re(m)|<1$ is called \emph{exceptional} if
$\pm \pi \in \Im\big(\frac1m\Ln(\varsigma)\big)$.
Similarly, a parameter $\nu\in \C$ is called \emph{exceptional} if
$\Im(\nu)=\pm \frac\pi2$.
\end{definition}

Let us immediately stress that $\pm \pi \in \Im\big(\frac1m\Ln(\varsigma)\big)$ are two independent conditions.
Indeed by setting $m= m_\r+ i m_\i \in \C^*$ with $m_\i,m_\r\in \R$ and $|m_\r|<1$ one has
\begin{align}\label{eq_alpha}
\nonumber \alpha \in \Im\Big(\frac1m\Ln(\varsigma)\Big) & \Leftrightarrow
\alpha \in \Im\Big(\frac1m\big(\ln(\varsigma)+2\pi i \Z\big)\Big) \\
\nonumber & \Leftrightarrow \alpha \in \Im\Big(\frac1m\ln(\varsigma)\Big) + 2\pi \Re\Big(\frac{1}{m}\Big)\Z \\
& \Leftrightarrow \alpha \in \Im\Big(\frac1m\ln(\varsigma)\Big) + 2\pi \frac{m_\r}{m_\r^2+m_\i^2}\Z.
\end{align}
Thus, this equation can either be satisfied for $\alpha=\pi$ or for $\alpha=-\pi$, or can be satisfied both for $\pi$ and $-\pi$.
Indeed, if $m_\r\neq 0$ both conditions are simultaneously satisfied if and only if the following system of equations
is satisfied for some $z_1,z_2 \in \Z$ with $z_1\neq z_2$~:
\begin{equation*}
\begin{cases}
& \Im\big(\frac1m\ln(\varsigma)\big)=-\pi\frac{z_1+z_2}{z_1-z_2} \\
& \frac{m_\r^2+m_\i^2}{m_\r}=z_1-z_2.
\end{cases}
\end{equation*}
On the other hand, if $m=in$ for some $n\in \R^*$ then \eqref{eq_alpha}
corresponds to $\alpha =-\big(\frac{1}{n}\ln(|\kappa|)\big)$.

Let us now recall some information about the point spectrum of the operators
$H_{m,\kappa}$ or $H_0^\nu$.

\begin{theorem}[Theorem 5.2 in \cite{DR}]\label{thm_spec}
Let $m\in \C$ with $|\Re(m)|<1$.
\begin{enumerate}
\item[(i)] For $m\neq 0$ one has $\sigma_\p(H_{m,0})=\emptyset$ while for $\kappa\in \C^*$ one has
\begin{equation*}\label{eq_align}
\sigma_{\p}(H_{m,\kappa})= \Big\{-4\e^{-w} \mid w\in
\frac{1}{m}\Ln(\varsigma) \hbox{ and } -\pi<\Im(w)<\pi\Big\}
\end{equation*}
\item[(ii)] For any $\nu \in \C$, $\sigma_\p(H_0^\nu)$ is nonempty if and only if $-\frac\pi2<\Im(\nu)<\frac\pi2$, and then
\begin{equation*}
\sigma_\p(H_0^\nu) =\big\{-4\e^{2(\nu-\gamma)}\big\}
\end{equation*}
where $\gamma$ denotes the Euler's constant.
\end{enumerate}
\end{theorem}

From this statement we can guess that $\sigma_\p(H_{m,\kappa})$ depends in a complicated way on
the parameters $m$ and $\kappa$.
There even exists a pattern of \emph{phase transitions}, when
some eigenvalues disappear in the continuous spectrum.
By looking carefully on the conditions appearing in the above statement
we can see that the exceptional situations correspond to the borderline cases,
and that the location of the eigenvalues are not arbitrary.
More precisely if we set $\C_\pm:=\{z\in \C\mid \pm \Im(z)>0\}$ then one has:

\begin{lemma}\label{lem_acumul}
\begin{enumerate}
\item[(i)] Let $(m_n)_{n\in\N}\subset \{z\in \C^*\mid |\Re(z)|<1\}$ and $(\kappa_n)_{n\in\N}\subset \C^*$
be two sequences, set
$$
\varsigma_n:=\kappa_n\frac{\Gamma(-m_n)}{\Gamma(m_n)}, \quad
a_n:=\Re\Big(\frac{1}{m_n}\ln(\varsigma_n)\Big), \quad
b_n:= \Im\Big(\frac{1}{m_n}\ln(\varsigma_n)\Big)
$$
and assume that $(a_n)_{n\in\N}$ converges to $a_\infty\in \R$ and that
$(b_n)_{n\in\N}$ is an increasing sequence converging to $\pi$.
Then, for $n$ large enough there exists $\lambda_n\in \sigma_\p(H_{m_n, \kappa_n})$ with
$\lambda_n \in \C_+$ and $\lambda_n\to 4\e^{-a_\infty}$ as $n\to \infty$.
If $(b_n)_{n\in\N}$ is a decreasing sequence converging to $-\pi$,
then for $n$ large enough $\lambda_n \in \C_-$ and $\lambda_n\to 4\e^{-a_\infty}$ as $n\to \infty$.
The value $4\e^{-a_\infty}$ is not an eigenvalue for any operator $H_{m,\kappa}$.
\item[(ii)] Let $(a_n)_{n\in \N}\subset \R$ be a sequence converging to $a_\infty\in \R$ and let
$(b_n)_{n\in\N} \subset (0,\frac{\pi}{2})$ be an increasing sequence converging to $\frac{\pi}{2}$.
Then, for any $n\in \N$ and for $\nu_n:=a_n+ib_n$ one has $\sigma_\p(H_0^{\nu_n})=\{\lambda_n\}\subset \C_-$
and $\lambda_n\to 4\e^{2(a_\infty-\gamma)}$ as $n\to \infty$.
If $(b_n)_{n\in\N} \subset (-\frac{\pi}{2},0)$ is a decreasing sequence converging to $-\frac{\pi}{2}$,
then one has $\sigma_\p(H_0^{\nu_n})=\{\lambda_n\}\subset \C_+$
and $\lambda_n\to 4\e^{2(a_\infty-\gamma)}$ as $n\to \infty$.
The value $4\e^{2(a_\infty-\gamma)}$ is not an eigenvalue for any operator $H_0^\nu$.
\end{enumerate}
\end{lemma}

\begin{proof}
The proof simply consists in an application of Theorem \ref{thm_spec}.
\end{proof}

Let us recall one more result related to eigenvalues. Later on, we shall need the following characterization
of $\#\sigma_\p(H_{m,\kappa})$, {\it i.e.}~of the  number of eigenvalues of $H_{m,\kappa}$.

\begin{proposition}[Proposition 5.3 in \cite{DR}]\label{prop_eigenv}
Let $m= m_\r+ i m_\i \in \C^*$ with $|m_\r|<1$, and let $\kappa \in \C^*$.
\begin{enumerate}
\item[(i)] If $m_\r=0$ and $\frac{\ln(|\kappa|)}{m_\i}\in (-\pi,\pi)$
then $\#\sigma_\p(H_{m,\kappa}) = \infty$,
\item[(ii)] If $m_\r=0$ and $\frac{\ln(|\kappa|)}{m_\i}\not \in (-\pi,\pi)$
then $\#\sigma_\p(H_{m,\kappa}) = 0$,
\item[(iii)]  If $m_\r\neq 0$ and if $N\in \{0,1,2,\dots\}$ satisfies
$N<\frac{m_\r^2+m_\i^2}{|m_\r|} \leq N+1$, then
\begin{equation*}
\#\sigma_\p(H_{m,\kappa})\in \{N,N+1\}.
\end{equation*}
\end{enumerate}
\end{proposition}

Let us now turn to the continuous spectrum for the operators $H_{m,\kappa}$ and $H_0^\nu$.
It has been shown in \cite{DR} that $[0,\infty)$ belongs to the spectrum of all
these operators. In addition, a limiting absorption principle has been exhibited.
Such a result corresponds to the existence of a boundary value of the resolvent on
$(0,\infty)$ when considered in some weighted Hilbert spaces.
For that purpose, we introduce for $t>0$ the weighted spaces $\H_t$ and
$\H_{-t}$ with $\H_t$ the domain of the operator $\langle R\rangle^t$ of multiplication by the function $r\mapsto \langle r \rangle^t \equiv (1+r^2)^{t/2}$
in $L^2(\R_+)$, and $\H_{-t}$ stands for its dual space.
We also recall the definition of the Bessel functions for dimension $1$ as introduced and motivated in \cite{DR}, namely
\begin{align*}
\rm{the}\ \emph{modified Bessel function for dimension $1$}&&\Ia_m(z):=\sqrt{\frac{\pi z}{2}} I_m(z),\\
\rm{the}\ \emph{MacDonald function for dimension $1$}&&\Ka_m(z):=\sqrt{\frac{2 z}{\pi}} K_m(z),\\
\rm{the}\ \emph{Bessel function for dimension $1$}&&\Ja_m(z):= \sqrt{\frac{\pi z}{2}} J_m(z),\\
\rm{the}\ \emph{Hankel function of the 1st kind for dimension $1$}&&\Ha_m^+(z):=\sqrt{\frac{\pi z}{2}} H^+_m(z), \\
\rm{the}\ \emph{Hankel function of the 2nd kind for dimension $1$}&&\Ha_m^-(z):=
\sqrt{\frac{\pi z}{2}} H^-_m(z), \\
\rm{the}\ \emph{Neumann function for dimension $1$}&&\Ya_m(z):= \sqrt{\frac{\pi z}{2}} Y_m(z),
\end{align*}
where $I_m$ is the modified Bessel function, $K_m$ is the MacDonald function,
$J_m$ is the Bessel function, $H_m^\pm$ are the Hankel function of the 1st kind and of the 2nd kind, and $Y_m$ is the Neumann function.

In the following statements, we recall the limiting absorption principle obtained in \cite{DR} and improve the statement in the exceptional situations.
The operators $H_{m,\kappa}$ and $H_0^\nu$ are considered separately, and we set
\begin{equation*}
R_{m,\kappa}(z):=(H_{m,\kappa}-z)^{-1}\qquad and \qquad
R_0^{\nu}(z):=(H_0^\nu-z)^{-1}
\end{equation*}
for their resolvents.
We also set when $\kappa\neq 0$
\begin{align}\label{Omega_1}
\nonumber \Omega_{m,\kappa}^\pm
:= \Big\{k\in\R_+ \mid & \ k^2 =  4\e^{-\Re\left(\frac{1}{m} \ln(\varsigma)\right)}\;\! \e^{2\pi \Im\left(\frac{1}{m}\right)z}
\ \hbox{for any}\ z\in \Z \ \hbox{satisfying}  \\
& \pm \pi = \Im\big(\frac1m (\ln(\varsigma)+2\pi i z)\big)\Big\}.
\end{align}
As a consequence of the observation made after Definition \ref{def_exceptional} the set $\Omega_{m,\kappa}^\pm$
is empty if $(m,\kappa)$ is not an exceptional pair, it consists of one single value if
$\pm \pi \in \Im\big(\frac1m (\ln(\varsigma)+2\pi i \Z)\big)$ and $m_\r\neq 0$, but it consists
of an infinite set if $m_\r=0$ and $\pm \pi = -\big(\frac{1}{n}\ln(|\kappa|)\big)$.

\begin{proposition}\label{prop_1}
Let $m\in \C^*$ with $|\Re(m)|<1$, let $\kappa\in \C$, and let $k>0$.
\begin{enumerate}
\item[(i)] If $(m,\kappa)$ is not an exceptional pair, then the boundary values of the resolvent
\begin{equation*}
R_{m,\kappa}(k^2\pm  i 0):=\lim_{\epsilon \searrow 0}  R_{m,\kappa}(k^2\pm  i \epsilon)
\end{equation*}
exist in the sense of operators from  $\H_t$
to $\H_{-t}$ for any $t>\frac{1}{2}$,
uniformly in $k$ on each compact subset of $\R_+$.
The kernel of $R_{m,\kappa}(k^2\pm i 0)$ is given for $0<r\leq s$ by
\begin{align*}
& R_{m,\kappa}(k^2\pm i0;r,s)\\
&= \frac{\pm i}{k\big(1-\varsigma \e^{\mp i \pi m}\big(\tfrac{k}{2}\big)^{2m}\big)}
\Big(\Ja_{m}(kr)- \varsigma\big(\tfrac{k}{2}\big)^{2m}\Ja_{-m}(kr)\Big)\Ha_m^\pm(ks)
\end{align*}
and the same expression with the role of $r$ and $s$ exchanged for $0<s<r$.
\item[(ii)] If $\pi \in \Im\big(\frac1m \Ln(\varsigma)\big)$ but $-\pi \not\in \Im\big(\frac1m\Ln(\varsigma)\big)$,
then the statement (i) holds for $R_{m,\kappa}(k^2-  i 0)$, while for $R_{m,\kappa}(k^2+i0)$ it only holds uniformly in
$k$ on each compact subset of $\R_+\setminus \Omega_{m,\kappa}^+$.
\item[(iii)] If $-\pi \in \Im\big(\frac1m \Ln(\varsigma)\big)$ but $\pi \not \in \Im\big(\frac1m\Ln(\varsigma)\big)$,
then the statement (i) holds for $R_{m,\kappa}(k^2+  i 0)$, while for $R_{m,\kappa}(k^2-i0)$ it only holds uniformly in
$k$ on each compact subset of $\R_+\setminus \Omega_{m,\kappa}^-$.
\item[(iv)] If $\pi \in \Im\big(\frac1m \Ln(\varsigma)\big)$  and
$-\pi \in \Im\big(\frac1m \Ln(\varsigma)\big)$,
then the statement (i) holds for $R_{m,\kappa}(k^2+i0)$ uniformly in $k$ on each
compact subset of $\R_+\setminus \Omega_{m,\kappa}^+$, and for $R_{m,\kappa}(k^2-i0)$ uniformly in
$k$ on each compact subset of $\R_+\setminus \Omega_{m,\kappa}^-$.
\end{enumerate}
\end{proposition}

For the next statement we set
\begin{equation*}
\Omega_0^\nu:= \big\{k\in\R_+ \mid k^2 = 4\e^{2(\Re(\nu)-\gamma)}\big\}.
\end{equation*}

\begin{proposition}\label{prop_2}
Let $\nu \in \C$, and let $k>0$.
\begin{enumerate}
\item[(i)] If $\nu$ is not an exceptional parameter, then the boundary values of the resolvent
\begin{equation*}
R_0^\nu(k^2\pm i 0):=\lim_{\epsilon \searrow 0}  R_{0}^\nu(k^2\pm i \epsilon)
\end{equation*}
exist in the sense of operators from $\H_t$
to $\H_{-t}$ for any $t>\frac{1}{2}$,
uniformly in $k$ on each compact subset of $\R_+$.
The kernel of $R_{0}^\nu(k^2\pm i 0)$ is given for $0<r\leq s$ by
\begin{align*}
& R_{0}^\nu(k^2\pm i0; r,s) \\
& = \frac{\pm i}{k\big(\gamma +\ln\big(\frac{k}{2}\big)-\nu\mp i \frac{\pi}{2}\big)}
\Big(\big(\gamma+\ln\big(\tfrac{k}{2}\big)- \nu \big)\Ja_0(kr)- \tfrac{\pi}{2} \Ya_0(kr)\Big)
\Ha_0^\pm(ks).
\end{align*}
and the same expression with the role of $r$ and $s$ exchanged for $0<s<r$.
\item[(ii)] If $\Im(\nu)=\frac\pi2$, then the statement (i) holds for $R_0^{\nu}(k^2+i0)$
while for $R_0^\nu(k^2-i0)$ it only holds uniformly in
$k$ on each compact subset of $\R_+\setminus \Omega_0^\nu$.
\item[(iii)] If $\Im(\nu)=-\frac\pi2$, then the statement (i) holds for $R_0^{\nu}(k^2-i0)$
while for $R_0^\nu(k^2+i0)$ it only holds uniformly in
$k$ on each compact subset of $\R_+\setminus \Omega_0^\nu$.
\end{enumerate}
\end{proposition}

\begin{proof}[Proof of Propositions \ref{prop_1} \& \ref{prop_2}]
In the non exceptional situations, these statements already appeared in \cite[Thm.~6.1 \& Prop.~7.1]{DR}
while in the exceptional situations the statements above are slightly more precise than the corresponding
ones in this reference. In fact, the only difference is that a more careful analysis of some numerical prefactors
is considered here. Namely, let us consider the following equivalences:
\begin{align*}
1-\varsigma \e^{\mp i \pi m}\big(\tfrac{k}{2}\big)^{2m} = 0
& \Leftrightarrow \varsigma = \e^{m \left(\pm i \pi - \ln(\frac{k}{2})^2\right)} \\
& \Leftrightarrow \frac{1}{m} \Ln(\varsigma) \ni \pm i \pi - \ln\Big(\frac{k}{2}\Big)^2 \\
& \Leftrightarrow
\begin{cases}
\pm \pi = \Im\big(\frac{1}{m} (\ln(\varsigma)+2\pi i z_\pm)\big) \\
k^2 = 4\e^{-\Re\left(\frac{1}{m} \ln(\varsigma)\right)}\;\! \e^{2\pi \Im\left(\frac{1}{m}\right)z_\pm}
\end{cases}
\end{align*}
whenever such $z_\pm\in \Z$ exist.
Then, the convergences mentioned in the statement only hold if the factor
$\big(1-\varsigma \e^{\mp i \pi m}\big(\tfrac{k}{2}\big)^{2m}\big)^{-1}$ does not vanish,
and the previous computation explains the necessary restriction in the exceptional situations.

For the second statement, it is sufficient to observe that
\begin{align*}
\gamma +\ln\Big(\frac{k}{2}\Big)-\nu\mp i \frac{\pi}{2} = 0
\Leftrightarrow
\begin{cases}
\Im(\nu)=\mp \frac{\pi}{2} \\
k^2 =  4\e^{2(\Re(\nu)-\gamma)}
\end{cases}
\end{align*}
and the same argument allows us to conclude.
\end{proof}

Before turning our attention to scattering theory, let us add two remarks related to the above limiting absorption principle:

\begin{remark}
The information provided on the discrete spectrum and on the continuous spectrum are quite consistent. Indeed, let us compare the content
of Lemma \ref{lem_acumul} with the previous two propositions.
If $(m,\kappa)$ is an exceptional pair, the limiting absorption principle
holds without limitation in the half-plane
in $\C$ where there is no possible accumulation of eigenvalues of some operators in the same family.
On the other hand, in the half-plane where there is a possible accumulation of eigenvalues
the limiting absorption principle holds only away from these singular points.
A similar observation is also valid for the operator $H_0^\nu$
when $\nu$ is an exceptional parameter.
\end{remark}

\begin{remark}\label{rem_sin}
The elements in the sets $\Omega^\pm_{m,\kappa}$ and $\Omega_0^\nu$ correspond to spectral singularities of the operators $H_{m,\kappa}$ and $H_0^\nu$ respectively, see  \cite[Sec.~2]{Sch} and \cite[Sec.~2.3]{FF} for more information on this concept. In our setting, it means that if $k_o \in \Omega^\pm_{m,\kappa}$ then the following
limits hold uniformly in $k$ on suitable neighborhood of $k_o$, namely
$$
\lim_{\epsilon \searrow 0} |k^2-k_o^2|\;\!R_{m,\kappa}(k^2\pm i \epsilon)
$$
exist in the sense of operators from $\H_t$ to $\H_{-t}$ for $t>\frac{1}{2}$.
Similar limits also hold for $k_o \in \Omega_0^\nu$ and for the resolvent of the operator $H_0^\nu$. Note finally that if $(m,\kappa)$ is an exceptional pair and if
$\Re(m)=0$, then $H_{m,\kappa}$ has an infinite number of spectral singularities.
\end{remark}

\section{Scattering theory}\label{sec_scat}

In this section we review the scattering theory for the operators $H_{m,\kappa}$ and $H_0^\nu$
as developed in \cite{DR}. However, exceptional situations were disregarded in this reference,
so we also provide new information about the corresponding operators.

First of all, let us recall the definition of the Hankel transform.
For any $m\in \C$ with $\Re(m)>-1$  we set
$\cF_{m}:C_{\rm c}(\R_+)\to L^2(\R_+)$ with
\begin{equation*}
\big(\cF_{m} f\big)(r):=\int_0^\infty\cF_{m} (r,s)f(s)\d s
\end{equation*}
and
\begin{equation*}
\cF_{m} (r,s) := \sqrt{\frac2\pi}\Ja_m(rs).
\end{equation*}
It has been shown in \cite[Prop.~4.5]{DR} that this map extends continuously
to a bounded invertible operator in $L^2(\R_+)$, with $\cF_m^{-1}=\cF_m$.
Additional information about this operator will be provided later on.

Based on this transformation, for any $m\in \C$ with $|\Re(m)|<1$ and for any
$\kappa\in \C$ let us define the \emph{incoming} and
\emph{outgoing Hankel transformations} $\cF_{m,\kappa}^\mp$ given by
\begin{equation*}
\cF_{m,\kappa}^\mp =
\Big(\cF_m - \varsigma  \cF_{-m}\big(\tfrac{R}{2}\big)^{2m}\Big)
\frac{\e^{\mp i \frac{\pi}{2}m}}{1-\varsigma\e^{\mp i \pi m}\big(\tfrac{R}{2}\big)^{2m}}.
\end{equation*}
As already seen in the proof of Propositions \ref{prop_1}, the denominator in the last factor vanishes only if $(m,\kappa)$ is an exceptional pair. However, even in this case
the operators $\cF_{m,\kappa}^\mp$ are still well defined as unbounded operator on several natural
domains.
For example, this operator is well defined on the set $C_{\rm c}\big(\R_+\setminus\Omega_{m,\kappa}^\pm\big)$
with $\Omega_{m,\kappa}^\pm$ introduced in \eqref{Omega_1}.

In order to get a better understanding of these operators, let us provide a slightly modified presentation of them.
For that purpose, we define the unitary and self-adjoint transformation
$J:L^2(\R_+)\to L^2(\R_+)$ by the formula
\begin{equation*}
\big(J f\big)(r)=\frac{1}{r}f\Big(\frac{1}{r}\Big)
\end{equation*}
for any $f\in L^2(\R_+)$ and $r\in \R_+$.
We also denote by $A$ the generator of dilation group in $L^2(\R_+)$, namely the generator
of the unitary group $\{U_\tau\}_{\tau\in \R}$ satisfying $[U_\tau f](r)=\e^{\tau/2}f(\e^\tau r)$ for any
$\tau\in \R$, $f\in L^2(\R_+)$ and $r\in \R_+$.
Finally we introduce the bounded and continuous function $\Xi_m:\R\to \C$ defined for $t\in \R$ by
\begin{equation}\label{def_Xi}
\Xi_m(t):= \e^{i\ln(2)t} \frac{\Gamma(\frac{m+1+i t}{2})}{\Gamma(\frac{m+1-i t}{2})}\ .
\end{equation}
We can now provide a slightly generalized version of \cite[Lem.~6.3]{DR}~:

\begin{lemma}\label{lem_kappa}
For any $m\in \C^*$ with $|\Re(m)|<1$ and any $\kappa\in \C$ the following equality holds:
\begin{equation}\label{eq_departure}
\cF_{m,\kappa}^\mp =  J \Big(\Xi_m(A)-\varsigma\Xi_{-m}(A) \big(\tfrac{R}{2}\big)^{2m}\Big)
\frac{\e^{\mp i \frac{\pi}{2}m}}{1-\varsigma \e^{\mp i \pi m} \big(\tfrac{R}{2}\big)^{2m}}.
\end{equation}
If $(m,\kappa)$ is not an exceptional pair, then this equality holds between bounded operators,
while if $(m,\kappa)$ is an exceptional pair, then:
\begin{enumerate}
\item[(i)] If $\pi \in \Im\big(\frac1m \Ln(\varsigma)\big)$ but $-\pi \not\in \Im\big(\frac1m\Ln(\varsigma)\big)$,
then $\cF_{m,\kappa}^+$ extends to a bounded operator while the above equality for $\cF_{m,\kappa}^-$
holds on $C_{\rm c}\big(\R_+\setminus\Omega_{m,\kappa}^+\big)$.
\item[(ii)] If $-\pi \in \Im\big(\frac1m \Ln(\varsigma)\big)$ but $\pi \not \in \Im\big(\frac1m\Ln(\varsigma)\big)$,
then $\cF_{m,\kappa}^-$ extends to a bounded operator while the above equality for $\cF_{m,\kappa}^+$
holds on $C_{\rm c}\big(\R_+\setminus\Omega_{m,\kappa}^-\big)$.
\item[(iii)] If $\pi \in \Im\big(\frac1m \Ln(\varsigma)\big)$ and
$-\pi \in \Im\big(\frac1m \Ln(\varsigma)\big)$,
the above equality holds for $\cF_{m,\kappa}^-$ on $C_{\rm c}\big(\R_+\setminus\Omega_{m,\kappa}^+\big)$
and for $\cF_{m,\kappa}^+$ on $C_{\rm c}\big(\R_+\setminus\Omega_{m,\kappa}^-\big)$.
\end{enumerate}
\end{lemma}

Similarly, for any $\nu \in \C$ we define the \emph{incoming} and \emph{outgoing Hankel transformations} $\cF_0^{\nu \mp}$ given by
the kernels for $r,s\in \R_+$~:
\begin{align*}
\cF_0^{\nu \mp}(r,s) & :=\sqrt{\frac2\pi}
\Big(\Ja_0(rs) \pm \frac{i \frac{\pi}{2}}{\gamma+ \ln\big(\frac{s}{2}\big)-\nu \mp i \frac{\pi}{2}} \;\!\Ha_0^\pm(rs)\Big)\\
& = \sqrt{\frac2\pi}
\bigg(\frac{\big(\gamma+ \ln\big(\frac{s}{2}\big)-\nu \big)\Ja_0(rs)- \frac{\pi}{2} \Ya_0(rs)}{\gamma + \ln\big(\frac{s}{2}\big)-\nu \mp i \frac{\pi}{2}}\bigg).
\end{align*}
However, a better understanding of these transformations can be obtained with the subsequent formulas:

\begin{lemma}\label{lem_nu}
For any $\nu \in \C$ the following alternative description of $\cF_0^{\nu \mp}$ hold:
\begin{equation}\label{eq_departure_2}
\cF_0^{\nu \mp} = J \Xi_0(A) \Big(\gamma +\ln\big(\tfrac{R}{2}\big)-\nu -i \tfrac{\pi}{2} \tanh\big(\tfrac{\pi}{2}A\big)\Big)
\frac{1}{\gamma + \ln\big(\frac{R}{2}\big)-\nu \mp i \frac{\pi}{2}}.
\end{equation}
If $\nu$ is not an exceptional parameter, then this equality holds between bounded operators,
while if $\nu$ is an exceptional pair, then:
\begin{enumerate}
\item[(i)] If $\Im(\nu)=\frac\pi2$, then $\cF_0^{\nu -}$ extends to a bounded operator while the above equality for $\cF_0^{\nu +}$
holds on $C_{\rm c}\big(\R_+\setminus\Omega_0^\nu\big)$.
\item[(ii)] If $\Im(\nu)=-\frac\pi2$, then $\cF_0^{\nu +}$ extends to a bounded operator while the above equality for $\cF_0^{\nu -}$
holds on $C_{\rm c}\big(\R_+\setminus\Omega_0^\nu\big)$.
\end{enumerate}
\end{lemma}

\begin{proof}[Proof of Lemmas \ref{lem_kappa} \& \ref{lem_nu}]
In the non exceptional cases, these statements and their proofs already appeared in \cite[Lem.~6.3 \& Corol.~7.6]{DR}.
It is then enough to observe that the same arguments hold in the exceptional cases, when the possible singularities
of the multiplication operators are taken into account by choosing suitable domains for these operators.
\end{proof}

Before introducing the wave operators, let us observe that the operators $\cF_{\frac{1}{2},0}^\mp$ take a very explicit form.
Indeed, as shown in \cite[Sec.~4.7]{DR} one has
\begin{equation*}
\cF_{\rm D}^\mp \equiv \cF_{\frac{1}{2},0}^\mp
= \e^{\mp i \frac{\pi}{4}} \Xi_{\frac{1}{2}}(-A)J
= \e^{\mp i \frac{\pi}{4}}\cF_{\rm D}
\end{equation*}
with
$$
\big(\cF_{\rm D} f\big)(r):=\sqrt{\frac2\pi}\int_0^\infty \sin(rs) f(s)\;\!\d s, \qquad f\in L^2(\R_+).
$$
This operator is clearly unitary.

We can now introduce the wave operators for the pairs $(H_{m,\kappa},H_{\rm D})$
or $(H_0^\nu,H_{\rm D})$, where $H_{\rm D}\equiv H_{\frac{1}{2},0}$ denotes the Dirichlet Laplacian on $\R_+$.
Note that $H_{\rm D}$ for the reference operator is chosen for simplicity, but other choices
are possible and lead to interesting phenomena, as emphasized in \cite{IR}.
In \cite{DR} the wave operators are defined by the formulas
\begin{equation*}
W_{m,\kappa}^\mp\equiv W^\mp(H_{m,\kappa},H_{\rm D}) := \cF_{m,\kappa}^{\mp}\;\! \cF_{\rm D}^{\pm}
\end{equation*}
and
\begin{equation*}
W_0^{\nu \mp} \equiv W^\mp(H_0^\nu,H_{\rm D}):= \cF_{0}^{\nu \mp}\;\! \cF_{\rm D}^{\pm}.
\end{equation*}
However, since some of these operators are unbounded in the exceptional cases, we shall use a unitarily equivalent definition
for these operators, namely we shall consider $\cF_{\rm D}^{\pm}\;\!W^\mp \;\!(\cF_{\rm D}^{\pm})^{-1}$, or more precisely
\begin{equation*}
\begin{split}
W_{m,\kappa}^\mp &:= \cF_{\rm D}^{\pm}\;\!\cF_{m,\kappa}^{\mp} \\
& = \Xi_{\frac{1}{2}}(-A) \Big(\Xi_m(A)-\varsigma\Xi_{-m}(A) \big(\tfrac{R}{2}\big)^{2m}\Big)
\frac{\e^{\mp i \frac{\pi}{2}(m-\frac{1}{2})}}{1-\varsigma \e^{\mp i \pi m} \big(\tfrac{R}{2}\big)^{2m}}
\end{split}
\end{equation*}
and
\begin{equation*}
\begin{split}
W_0^{\nu \mp} &:= \cF_{\rm D}^{\pm}\;\!\cF_{0}^{\nu \mp} \\
& =\Xi_{\frac{1}{2}}(-A) \Xi_0(A) \Big(\gamma +\ln\big(\tfrac{R}{2}\big)-\nu -i \tfrac{\pi}{2} \tanh\big(\tfrac{\pi}{2}A\big)\Big)
\frac{\e^{\pm i \frac{\pi}{4}}}{\gamma + \ln\big(\frac{R}{2}\big)-\nu \mp i \frac{\pi}{2}}.
\end{split}
\end{equation*}

As a direct consequence of Lemmas \ref{lem_kappa} and \ref{lem_nu} most of these operators
are bounded and thus well defined on $L^2(\R_+)$. However, in the exceptional cases
some of the operators $\cF_{m,\kappa}^{\pm}$ are not bounded, and thus were only defined on suitable domains.
For that reason the corresponding wave operators are also only defined on the same domains.
For completeness, let us enumerate the operators which are unbounded, and consequently which require a special
attention for their definition:
\begin{enumerate}
\item[(i)] If $\pi \in \Im\big(\frac1m \Ln(\varsigma)\big)$ but $-\pi \not\in \Im\big(\frac1m\Ln(\varsigma)\big)$,
then $W_{m,\kappa}^-$ is unbounded.
\item[(ii)] If $-\pi \in \Im\big(\frac1m \Ln(\varsigma)\big)$ but $\pi \not \in \Im\big(\frac1m\Ln(\varsigma)\big)$,
then $W_{m,\kappa}^+$ is unbounded.
\item[(iii)] If $\pi \in \Im\big(\frac1m (\Ln(\varsigma)\big)$ and
$-\pi \in \Im\big(\frac1m \Ln(\varsigma)\big)$, then
both operators $W_{m,\kappa}^\pm$ are unbounded.
\item[(iv)] If $\Im(\nu)=\frac\pi2$, then $W_0^{\nu +}$ is unbounded.
\item[(v)] If $\Im(\nu)=-\frac\pi2$, then $W_0^{\nu -}$ is unbounded.
\end{enumerate}
Except the operators appearing in the above list, all wave operators are bounded.

\begin{remark}\label{rem_1}
In the previous two lemmas and in the above statement we have been rather pessimistic, and there is a tiny
chance that the situation is slightly better than described. Indeed, if one looks carefully
at the expressions provided in \eqref{eq_departure} and \eqref{eq_departure_2} the operator $\cF_{m,\kappa}^\mp$
or $\cF_0^{\nu \mp}$ consist in a product of two types of operators, namely
some functions of $A$ and some multiplication operators (functions of $R$).
For all parameters $m$, $\kappa$ and $\nu$ the functions of $A$ are bounded.
On the other hand, depending on the parameters $m$, $\kappa$ and $\nu$ the multiplication operators
are either bounded or not. In the latter case, the operators $W_{m,\kappa}^\mp$ or $W_0^{\nu \mp}$
can be defined on a natural domain for the unbounded multiplication operators,
but it is not clear if this domain can be extended due to some cancellations with the 
other factors. In the above statements, we took the precautious attitude of not expecting
any improvement, and for that reason we mentioned that some wave operators are unbounded. 
It would certainly be interesting to further investigate in this direction and get
a better description of the maximal domain of these operators and of their range.
So far, our attempts have not been successful. 
\end{remark}

\section{Index theorems}\label{sec_index}

In this section we provide the algebraic framework which leads to a topological
version of Levinson's theorem. This framework for the current model already appeared in \cite{IR,NPR}, but we extend the results presented in these references in three directions. Indeed, we shall consider systems with arbitrary eigenvalues (complex or real) and in arbitrary number (finite or infinite). In the former reference, only real eigenvalues were considered (which means only self-adjoint operators $H_{m,\kappa}$ were studied), and in the latter only a finite number of complex eigenvalues were considered. In addition, we also provide an index theorem for the pair $(H_0^\nu, H_{\rm D})$ which has never been exhibited before.

Our first task is to provide a more familiar but unitarily equivalent representation of the wave operators. Indeed, since the operators $A$ and $B:=\ln\big(\frac{R}{2}\big)$ in $L^2(\R_+)$ satisfy the Weyl commutation relation, they are unitarily equivalent to the operators $D=-i\partial_x$ and $X$ in $L^2(\R)$. This equivalence is essentially implemented by a Mellin transform. Through this transformation the wave operators introduced in the previous section are given by the following expressions:
\begin{equation}\label{eq_11}
\W_{m,\kappa}^\mp := \Xi_{\frac{1}{2}}(-D) \Big(\Xi_m(D)-\varsigma\Xi_{-m}(D) \e^{2mX}\Big)
\frac{\e^{\mp i \frac{\pi}{2}(m-\frac{1}{2})}}{1-\varsigma \e^{\mp i \pi m} \e^{2mX}}
\end{equation}
and
\begin{equation}\label{eq_12}
\W_0^{\nu \mp} := \Xi_{\frac{1}{2}}(-D) \Xi_0(D) \Big(\gamma +X-\nu -i \tfrac{\pi}{2} \tanh\big(\tfrac{\pi}{2}D\big)\Big)
\frac{\e^{\pm i \frac{\pi}{4}}}{\gamma + X-\nu \mp i \frac{\pi}{2}}.
\end{equation}
This representation is more familiar since these operators correspond now to
pseudo-differential operators.

Our second task is to observe that these operators are made of functions of $D$ and
$X$ which have precise properties.
For that purpose, let us set when $\kappa\neq 0$
\begin{align*}
\nonumber \Lambda_{m,\kappa}^\pm
:= \Big\{x\in\R \mid & \ x =  -\frac{1}{2}\Big(\Re\big(\frac{1}{m} \ln(\varsigma)\big) -2\pi \Im\big(\frac{1}{m}\big)z\Big)
\ \hbox{for any}\ z\in \Z \ \hbox{satisfying}  \\
& \pm \pi = \Im\big(\frac1m (\ln(\varsigma)+2\pi i z)\big)\Big\}
\end{align*}
and $\Lambda_0^\nu :=\emptyset$ if $\Im(\nu)\neq \pm \frac{\pi}{2}$  while if $\Im(\nu)=\pm \frac{\pi}{2}$
$$
\Lambda_0^\nu:= \big\{x\in\R \mid x = \Re(\nu)-\gamma\big\}.
$$
These sets are the counterparts of $\Omega_{m,\kappa}^\pm$ and $\Omega_0^\nu$
in the new representation.
We then define the functions of two variables:
$\Gamma_{m,\kappa}^\mp : \R\setminus \Lambda_{m,\kappa}^{\pm} \times \R\to \C$ and $\Gamma_0^{\nu \mp}:\R\setminus \Lambda_0^\nu \times \R\to \C$ by
\begin{equation*}
\Gamma_{m,\kappa}^\mp(x,\xi) := \Xi_{\frac{1}{2}}(-\xi) \Big(\Xi_m(\xi)-\varsigma\Xi_{-m}(\xi) \e^{2mx}\Big)
\frac{\e^{\mp i \frac{\pi}{2}(m-\frac{1}{2})}}{1-\varsigma \e^{\mp i \pi m} \e^{2mx}}
\end{equation*}
and
\begin{equation*}
\Gamma_0^{\nu \mp} (x,\xi):= \Xi_{\frac{1}{2}}(-\xi) \Xi_0(\xi) \Big(\gamma +x-\nu -i \tfrac{\pi}{2} \tanh\big(\tfrac{\pi}{2}\xi\big)\Big)
\frac{\e^{\pm i \frac{\pi}{4}}}{\gamma + x-\nu \mp i \frac{\pi}{2}}.
\end{equation*}
Formally the following equalities hold:
$$
\W_{m,\kappa}^\mp = \Gamma_{m,\kappa}^\mp(X,D)\quad \hbox{ and }\quad
\W_0^{\nu \mp} = \Gamma_0^{\nu \mp} (X,D),
$$
but the only precise meaning is the one provided in \eqref{eq_11} and \eqref{eq_12}.

Let us now introduce the commutative algebra $C\big([-\infty,\infty]\big)$ of continuous functions on $\R$ having limits at $\pm \infty$.
We then recall from the proof of \cite[Thm.~4.10]{DR} that for any $m,m'\in \C$ with $\Re(m)>-1$ and $\Re(m')>-1$
the map $\xi\mapsto \Xi_m(-\xi)\Xi_{m'}(\xi)$ belongs to $C\big([-\infty,\infty]\big)$
and that the following equalities hold:
\begin{equation*}
\Xi_m(\mp \infty)\Xi_{m'}(\pm \infty)=\e^{\mp i \frac{\pi}{2}(m-m')}.
\end{equation*}
We also introduce two non-commutative algebras which are going to nest the wave operators $\W_{m,\kappa}^\mp$ and $\W_0^{\nu \mp}$.
Firstly, we consider the unital $C^*$-subalgebra of $\B\big(L^2(\R)\big)$
$$
\E_o=C^*\Big(a(D)b(X)\mid a\in C\big([-\infty,\infty]\big), b\in C\big([-\infty,+\infty]\big) \Big).
$$
Secondly, for any $n>0$ we introduce the unital $C^*$-subalgebra of $\B\big(L^2(\R)\big)$
\begin{equation*}
\E_n:=C^*\Big(a(D)b(X)\mid a\in C\big([-\infty,+\infty]\big),b\in C_{\frac{\pi}{n}}(\R) \Big),
\end{equation*}
where $C_{\frac{\pi}{n}}(\R)$ denotes the set of all continuous periodic functions on $\R$ with period $\frac{\pi}{n}$.

Based on a careful analysis of the functions $\Gamma_{m,\kappa}^\mp$ and $\Gamma_0^{\nu \mp}$ one easily deduces the following statement, see \cite{IR} and \cite{NPR} for similar results.

\begin{proposition}\label{prop_enumeration}
Let $m\in \C^*$ with $|\Re(m)|<1$, and let $\kappa, \nu \in \C$.
The following operators belong to $\E_o$:
\begin{enumerate}
\item[I.1)] $\W_{m,\kappa}^\mp$ if $(m,\kappa)$ is not an exceptional pair
and $\Re(m)\neq 0$,
\item[I.2)] $\W_{m,\kappa}^+ $ if $\Re(m)\neq 0$ and $\pi \in \Im\big(\frac1m \Ln(\varsigma)\big)$ but $-\pi \not\in \Im\big(\frac1m\Ln(\varsigma)\big)$,
\item[I.3)] $\W_{m,\kappa}^-$ if $\Re(m)\neq 0$ and $-\pi \in \Im\big(\frac1m \Ln(\varsigma)\big)$ but $\pi \not \in \Im\big(\frac1m\Ln(\varsigma)\big)$,
\item[I.4)] $\W_0^{\nu \mp} $ if $\nu$ is not an exceptional parameter,
\item[I.5)] $\W_0^{\nu -}$ if $\Im(\nu)=\frac\pi2$, and $\W_0^{\nu +}$ if $\Im(\nu)=-\frac\pi2$.
\end{enumerate}
The following operators belong to $\E_{|n|}$:
\begin{enumerate}
\item[II.1)]  $\W_{m,\kappa}^\mp$ if $(m,\kappa)$ is not an exceptional pair
and  $m=in$ for some $n \in \R^*$,
\item[II.2)] $\W_{m,\kappa}^+ $ if $m=in$ for some $n\in \R^*$ and $\pi \in \Im\big(\frac1m \Ln(\varsigma)\big)$ but $-\pi \not\in \Im\big(\frac1m\Ln(\varsigma)\big)$,
\item[II.3)] $\W_{m,\kappa}^-$ if  $m= in$ for some $n\in \R^*$ and $-\pi \in \Im\big(\frac1m \Ln(\varsigma)\big)$ but $\pi \not \in \Im\big(\frac1m\Ln(\varsigma)\big)$.
\end{enumerate}
In all other cases, the wave operators are not bounded and do no belong to any $C^*$-algebras.
\end{proposition}

\begin{remark}\label{rem_2}
When an operator is unbounded, it can not belong to any $C^*$-algebra
but it is still possible that its resolvent belong it.
It is thus a natural question to check if the unbounded wave operators
belong to the $C^*$-algebras introduced above. 
Since some essential information on these operators are still missing, 
as already mentioned in Remark \ref{rem_1}, we can not answer this question.
\end{remark}

Once we know that the wave operators belong to very explicit $C^*$-subalgebra of $\B\big(L^2(\R)\big)$, the third task consists in studying these algebras and their structures, and to deduce a suitable $C^*$-algebra framework for deducing index theorems.
The two algebras $\E_o$ and $\E_n$ will be studied independently, and we shall start with the former one.

The key observation for the analysis of $\E_o$ is that the ideal of compact operators $\K_\R:=\K\big(L^2(\R)\big)$ corresponds to the $C^*$-algebra generated
by products of the form $a(D)b(X)$ with $a,b\in C_0(\R)$,
with $C_0(\R)$ the algebra of continuous functions on $\R$ vanishing at $\pm \infty$.
Then, one easily infers that  $\E_o/\K_{\R}$ is isomorphic to $C(\square)$,
the algebra of continuous functions on the boundary $\square$
of the closed square $\blacksquare$.
Note that the unital quotient morphism $q_o:\E_o\to C(\square)$ is uniquely determined by
$$
q_o\big(a(D)b(X)\big) = \big(a(\cdot)b(-\infty),\;\! a(-\infty)b(\cdot),\;\!a(\cdot)b(+\infty),\;\! a(+\infty)b(\cdot)\big).
$$
In fact, the above notation corresponds to an embedding of the algebra $C(\square)$ as a subalgebra of
\begin{equation*}
C\big([-\infty,+\infty]\big)\oplus C\big([-\infty,+\infty]\big)\oplus C\big([-\infty,+\infty]\big)\oplus C\big([-\infty,+\infty]\big)
\end{equation*}
given by elements $(\Gamma_1, \Gamma_2, \Gamma_3, \Gamma_4)$ which coincide at the corresponding end points, that
is,
$\Gamma_1(-\infty) = \Gamma_2(-\infty)$, $\Gamma_2(+\infty) = \Gamma_3(-\infty)$, $\Gamma_3(+\infty) = \Gamma_4(+\infty)$, and
$\Gamma_4(-\infty) = \Gamma_1(+\infty)$.

Now, whenever the wave operators $\W_{m,\kappa}^\mp$ and $\W_0^{\nu \mp}$
belong to $\E_o$, the corresponding functions $\Gamma_{m,\kappa}^\mp$ or $\Gamma_0^{\nu \mp}$ admit a restriction to $\square$. More precisely in such a situation one easily deduces the following expressions for the restriction of $\Gamma_{m,\kappa}^-$ to $\square$~:
\begin{align*}
\Gamma_{m,\kappa;1}^-(\xi) & = \left\{\begin{matrix}  \e^{i\frac{\pi}{2}(\frac{1}{2}-m)}  \Xi_{\frac{1}{2}}(-\xi) \Xi_{m}(\xi)
& \hbox{ if } \ \Re(m)>0, \\
\e^{i\frac{\pi}{2}(\frac{1}{2}+m)}  \Xi_{\frac{1}{2}}(-\xi) \Xi_{-m}(\xi) & \hbox{ if } \ \Re(m)<0,
\end{matrix}\right.\\
\Gamma_{m,\kappa;2}^-(x) & =
\e^{i\pi(\frac{1}{2}-m)} \frac{1 -\varsigma \e^{+i \pi m} \e^{2mx} }{1-\varsigma \e^{- i \pi m} \e^{2mx}}, \\
\Gamma_{m,\kappa;3}^-(\xi) & = \left\{\begin{matrix} \e^{i\frac{\pi}{2}(\frac{1}{2}+m)}  \Xi_{\frac{1}{2}}(-\xi) \Xi_{-m}(\xi) & \hbox{ if } \ \Re(m)>0, \\
\e^{i\frac{\pi}{2}(\frac{1}{2}-m)}  \Xi_{\frac{1}{2}}(-\xi) \Xi_m(\xi)  & \hbox{ if } \ \Re(m)<0,
\end{matrix}\right.\\
\Gamma_{m,\kappa;4}^-(x) & = 1,
\end{align*}
and in the special case $\kappa=0$ one has $ \Gamma_{m,0;1}^-(\xi)= \Gamma_{m,0;3}^-(\xi)=\e^{i\frac{\pi}{2}(\frac{1}{2}-m)}  \Xi_{\frac{1}{2}}(-\xi) \Xi_m(\xi)$,
$ \Gamma_{m,0;2}^-(x)= \e^{i\pi(\frac{1}{2}-m)}$ and $ \Gamma_{m,0;4}^-(x)=1$.
Similarly, one has for the restriction of  $\Gamma_0^{\nu -}$ to $\square$
\begin{align*}
\Gamma_{0;1}^{\nu-}(\xi) & = \e^{i\frac{\pi}{4}}  \Xi_{\frac{1}{2}}(-\xi) \Xi_{m}(\xi) \\
\Gamma_{0;2}^{\nu-}(x) & =
\e^{i\frac{\pi}{4}} \frac{\gamma + x-\nu + i \frac{\pi}{2}}{\gamma + x-\nu - i \frac{\pi}{2}}, \\
\Gamma_{0;3}^{\nu-}(\xi) & = \e^{i\frac{\pi}{4}}  \Xi_{\frac{1}{2}}(-\xi) \Xi_{m}(\xi) \\
\Gamma_{0;4}^{\nu-}(x) & = 1,
\end{align*}
Note also that similar expressions for $\Gamma_{m,\kappa}^+$ and $\Gamma_0^{\nu +}$ can be computed, but they are not presented for brevity.

Before stating our first index theorem with the data mentioned above, two more information coming from scattering theory are necessary. The first one is related to the scattering operator. Let us recall that for a scattering system defined by the pair of operators $(H_{m,\kappa},H_{\rm D})$ the scattering operator $S_{m,\kappa}$ is given by the product
$W_{m,\kappa}^{-\#} \;\!W_{m,\kappa}^-$ where ${}^{\#}$ means the transpose operator.
Note that in  the self-adjoint case this definition corresponds to the more usual product $W_{m,\kappa}^{+*}W_{m,\kappa}^-$ and extends it when the operators are not self-adjoint. The scattering operator is known to commute with the reference operator, namely $H_{\rm D}$. In addition, this operator is unitarily equivalent to the multiplication operator in $L^2(\R)$ defined by the function $\Gamma_{m,\kappa;2}^-$. By analogy this operator will also be denoted by $\SS_{m,\kappa}$.
Clearly, a similar relation exists between the scattering operator $S_0^\nu$ for the pair $(H_0^\nu,H_{\rm D})$ and the multiplication operator in $L^2(\R)$ defined by the function $\Gamma_{0;2}^{\nu -}\equiv \SS_0^\nu$.
A key observation about these scattering operators is that whenever $(m,\kappa)$ or $\nu$ are not exceptional the functions $\Gamma_{m,\kappa;2}^-$ or  $\Gamma_{0;2}^{\nu -}$ are bounded and boundedly invertible on $\R$.
In these situations it then follows that the functions $\Gamma_{m,\kappa}^-$ and $\Gamma_0^{\nu -}$ defined on $\square$ are also invertible and boundedly invertible.
As a consequence, their winding numbers\footnote{Recall that for a function $f$ defined on a closed curve $\gamma$ and taking
values in $\C^*$ its winding number corresponds to number of times the function $t\mapsto f(t)$ turns around $0\in \C$
when $t$ runs on $\gamma$.} are well defined and will be denoted by $\wn$.
On the other hand, it is also easily observed that in the exceptional situations these functions are either not bounded or have an inverse which is not bounded.

The second necessary information is about the wave operators themselves. It is shown
in \cite{DR} that if $(m,\kappa)$ is not an exceptional pair the kernels of $W_{m,\kappa}^\mp$
are empty while the cokernels of these operators corresponds to the subspaces spanned by the eigenfunctions associated to the eigenvalues of $H_{m,\kappa}$. Similarly, if $\nu$ is not an exceptional parameter, the wave operators $W_0^{\nu \mp}$ have an empty kernel and a cokernel equal to the subspace spanned by the eigenfunctions of $H_0^\nu$.

With all the information collected so far, the following statement can easily be proved. It relies on the index map associated to the short exact sequence
$$
0\to \K_\R\to \E_o\to C(\square)\to 0
$$
and to the fact that $\W_{m,\kappa}^-$ is a lift for the invertible operator $\Gamma_{m,\kappa}^-\in C(\square)$ (and similarly $\W_{0}^{\nu -}$ is a lift for the invertible operator $\Gamma_0^{\nu -}\in C(\square)$). The details are provided in \cite{NPR}.
Note that the following statement applies for the cases I.1) and I.4) of Proposition
\ref{prop_enumeration}.
Let us still recall for clarity that the index of a Fredholm operator corresponds to the difference between the dimension of its kernel
and of its cokernel. This index will be denoted by $\Index$ in the sequel.

\begin{theorem}[Topological Levinson's theorem]\label{thm_1}
Let $m\in \C^*$ with $|\Re(m)|\in (0,1)$, and let $\kappa, \nu \in \C$.
If $(m,\kappa)$ and $\nu$ are not exceptional parameters, then the following relations hold:
$$
\wn\big[\Gamma_{m,\kappa}^-\big] = \hbox{ number of eigenvalues of }H_{m,\kappa}
$$
and
$$
\wn\big[\Gamma_{0}^{\nu -}\big] = \hbox{ number of eigenvalues of }H_{0}^\nu ,
$$
where the r.h.s.~are also equal to $-\Index(\W_{m,\kappa}^-)$ and $-\Index(\W_{0}^{\nu -})$, respectively.
 \end{theorem}

Note that the l.h.s.~of the above statement contains four contributions, one for each function living on the edges of the square. As already mentioned, the contribution of $\Gamma_2$ corresponds to the one of the scattering operator. In addition, the contribution due to $\Gamma_1$ and to $\Gamma_3$ correspond to \emph{corrections} to Levinson's theorem. Explanations on these corrections have been provided in \cite{R} and are quite common in any statement about Levinson's theorem. In our approach, these corrections are automatically taken into account.

\begin{remark}
The wave operators described in the cases I.2), I.3) and I.5) of Proposition \ref{prop_enumeration} also belong to $\E_o$ and the corresponding functions
$\Gamma_{m,\kappa}^-$ or $\Gamma_0^{\nu -}$ are well defined.
However, since these functions vanish at one point on $\square$ their winding
numbers are no more well-defined. Accordingly, the corresponding
wave operators are not Fredholm operators, and thus their analytic indexes are also
not well defined.
\end{remark}

Let us now turn our attention to the algebra $\E_n$ for $n>0$. Clearly, this
algebra does not contain $\K_\R$, and thus the previous construction does not apply.
In fact, this algebra contains all pseudo-differential operators of order $0$ with periodic coefficients.
In such a case the ideal $\K_\R$ has to be replaced by the ideal $\J_n$ defined by
\begin{equation*}
\J_n:=C^*\Big(a(D)b(X)\mid a\in C_0(\R),b\in C_{\frac{\pi}{n}}(\R) \Big).
\end{equation*}
Then, the quotient algebra $\E_n/\J_n$ can easily be computed and is isomorphic to $C_{\frac{\pi}{n}}(\R) \oplus C_{\frac{\pi}{n}}(\R)$.
The quotient morphism $q_n:\E_n\to  C_{\frac{\pi}{n}}(\R) \oplus C_{\frac{\pi}{n}}(\R) $ is uniquely determined by
$$
q_n\big(a(D)b(X)\big) = \big(a(-\infty)b(\cdot),\;\! a(+\infty)b(\cdot)\big).
$$

Now, let us consider $n\in \R^* $ and compute the image of $\W_{m,\kappa}^-$ by this quotient
map whenever $\W_{m,\kappa}^-$ belongs to $\E_{|n|}$, namely in the cases II.1) - II.3)
of Proposition \ref{prop_enumeration}. More precisely, for the operator $\W_{in,\kappa}^-$ one has
$$
q_{|n|}(\W_{in,\kappa}^-) =
\Big(
i\e^{\pi n} \frac{1 -\varsigma \e^{- \pi n} \e^{2inx} }{1-\varsigma \e^{+ \pi n} \e^{2inx}},1\Big).
$$
Note then that since $C_{\frac{\pi}{n}}(\R)$ can naturally be identified with $C(\S)$,
we define through this identification the winding number $\wn_{\frac{\pi}{n}}(f)$ of any bounded and boundedly invertible element
$f\in C_{\frac{\pi}{n}}(\R)$. Clearly, this is well defined if and only if $(in,\kappa)$ is not an exceptional pair.

In order to define an analytic index for $\W^-_{in,\kappa}$ we recall the construction provided in \cite[Sec.~4]{IR} about  the direct integral decomposition of $L^2(\R)$ useful for periodic systems, the so-called \emph{Floquet-Bloch decomposition}. More information can also be found in \cite[Sec.XIII.16]{RS}.
For simplicity, we provide only the construction for $n>0$, but the general
case can be obtained by replacing $n$ with $|n|$.
For each $\theta\in [0,2n)$ we set $\H_\theta:=L^2\big([0,\frac{\pi}{n}], \d x\big)$ endowed with the usual Lebesgue measure, and also define
\begin{equation*}
\H_n:=\int_{[0,2n)}^{\oplus}\H_\theta \frac{\d\theta}{2n}.
\end{equation*}
Then, if $\SS(\R)$ denotes the Schwartz space on $\R$, the map $\U_n:L^2(\R)\to\H_n$ defined for $\theta\in [0,2n)$ and $x\in[0,\frac{\pi}{n})$ by
\begin{equation*}
[\U_nf](\theta,x):=\sum_{k\in\Z}\e^{-i\frac{\pi}{n}k\theta}f\Big(x+\frac{\pi}{n}k\Big) \qquad \forall f \in \SS(\R),
\end{equation*}
extends continuously to a unitary operator.
The adjoint operator is then given by the formula
$$
[\U_n^*\varphi]\big(x+\frac{\pi}{n}k\big) = \int_0^{2n}\e^{i\frac{\pi}{n}k\theta}\varphi(\theta,x)\frac{\d \theta}{2n}.
$$
Moreover, one has
\begin{equation*}
\U_n D\U_n^*=\int_{[0,2n)}^{\oplus} D_x^{(\theta)} \frac{\d\theta}{2n},
\end{equation*}
where $D_x^{(\theta)}$ is the operator $-i\frac{\d}{\d x}$ on a fiber $\H_\theta$
with boundary condition $f(\frac{\pi}{n})=\e^{i\frac{\pi}{n} \theta}f(0)$.

Thus, for any operator of the form $a(D)b(X)$ with $a\in C_0\big([-\infty,\infty]\big)$
and $b\in C_{\frac{\pi}{n}}(\R)$, the operator $\U_na(D)b(X)\U_n^*$ is a decomposable operator with the fibers $a\big(D_x^{(\theta)}\big)b(X)$.
On suitable bounded decomposable operator $\Phi=\int_{[0,2n)}^{\oplus}\Phi(\theta) \frac{\d\theta}{2n}$
we also define the trace $\Trace_n$ by
\begin{equation*}
\Trace_n(\Phi)=\int_{0}^{2n}\trace_\theta\big(\Phi(\theta)\big) \frac{\d\theta}{2n}
\end{equation*}
where $\trace_\theta$ is the usual trace on $\H_\theta$.

Before stating our main result for the semi-Fredholm operator $\W_{in,\kappa}^-$, let us recall that $\W_{m,\kappa}^{\mp \#}$ denote the transpose operators of $\W_{m,\kappa}^\mp$, and that the following relations have been proved in \cite{DR} in the non exceptional case:
$$
\W_{m,\kappa}^{\pm \#} \W_{m,\kappa}^\mp = \one\qquad \hbox{ and }
\qquad  \W_{m,\kappa}^\mp \W_{m,\kappa}^{\pm \#} = \I_{\R_+}(H_{m,\kappa})
$$
where $\I_{\R_+}(H_{m,\kappa})$ is a projection related to the continuous spectrum
of $H_{m,\kappa}$. More precisely, the subspace spanned by this projection
is the image through the unitary transformation of the complementary to the subspace spanned by the eigenfunctions of the operator $H_{m,\kappa}$.
This latter subspace for $m=in$ is either infinite dimensional, or $0$-dimensional, as already mentioned in Proposition \ref{prop_eigenv}.
If we set $\I_{\rm p}(H_{m,\kappa}):=1-\I_{\R_+}(H_{m,\kappa})$
then one has:

\begin{theorem}\label{thm_2}
Consider $n>0$ and $\kappa \in \C$ such that $(in,\kappa)$ is not an exceptional pair.  Then,
\begin{equation*}
\wn_{\frac{\pi}{n}}\big(\SS_{in,\kappa}\big)
=-\Trace_n \big(\I_{\rm p}(H_{in,\kappa})\big).
\end{equation*}
\end{theorem}

Let us emphasize that the l.h.s.~corresponds to the natural analytic index $\Index_n$
defined in terms of $\Trace_n$ and evaluated on $\W_{in,\kappa}^-$.
A proof for such a statement is provided in \cite{IR} but is valid only if
$H_{in,\kappa}$ is self-adjoint. This takes place if and only if $|\kappa|=1$.
For completeness we provide below an adaptation of the proof valid in the more general context of the present paper. We shall show in this proof that the above equality can only take two values:
either $-1$ when $H_{in,\kappa}$ has an infinite number of eigenvalues, or $0$ when this operator has no eigenvalue.

\begin{proof}
In this proof we assume that $\kappa\neq 0$ since in this case the statement is trivially satisfied.
From the equalities
$$
|\varsigma|=\Big|\kappa \frac{\Gamma(-in)}{\Gamma(in)}\Big|= \e^{n(\frac{1}{n}\ln(|\kappa|)}
$$
together with the equality
\begin{equation*}
\SS_{in,\kappa}(x)=i\e^{\pi n}\frac{1-\varsigma \e^{ -\pi n}\e^{2inx}}{1-\varsigma \e^{ \pi n}\e^{2inx}}
\end{equation*}
one easily infers that $\wn_{\frac{\pi}{n}}\big(\SS_{in,\kappa}\big)=-1$
if $\frac{1}{n}\ln(|\kappa|) \in (-\pi,\pi)$ while  one has
$\wn_{\frac{\pi}{n}}\big(\SS_{in,\kappa}\big) =0$ if
$\frac{1}{n}\ln(|\kappa|) \not \in [-\pi,\pi]$. Since the special case
$\frac{1}{n}\ln(|\kappa|)= \pm \pi$ is an exceptional situation,
it is disregarded.

Let us now consider an operator of the form $a(D)b(X)$ with $a\in C_0(\R)$
and $b\in C_{\frac{\pi}{n}}(\R)$, and the corresponding operator $a\big(D_x^{(\theta)}\big)b(X)$.
Since the eigenfunctions of the operator $D_x^{(\theta)}$ are provided by the functions
$$
[0,\frac{\pi}{n})\ni x \mapsto \sqrt{{\frac{n}{\pi}}}\e^{i(\theta + 2nk)x}\in \C, \qquad k \in \Z
$$
we infer that the Schwartz kernel of the operator $a\big(D_x^{(\theta)}\big)b(X)$ is given by
\begin{equation*}
K_{a(D_x^{(\theta)})b(X)}(x,y)=\frac{n}{\pi}\sum_{k\in\Z}a(\theta+2nk) \e^{i(\theta+2nk)(x-y)} b(y).
\end{equation*}
Thus, if $b(X)a\big(D_x^{(\theta)}\big)$ is
$\trace_\theta$-trace class for a.e.~$\theta \in [0,2n)$ we obtain
\begin{equation*}
\trace_\theta\big(a\big(D_x^{(\theta)}\big)b(X)\big)
=\sum_{k\in\Z}a(\theta+2nk)\times \frac{n}{\pi}\int_{0}^{\frac{\pi}{n}}b(x)\;\!\d x
\end{equation*}
and then
\begin{equation}\label{formula of tracen}
\Trace_n\big(a(D)b(X)\big)=\frac{1}{2n}\int_{\R}a(\xi)\;\!\d\xi\times \frac{n}{\pi}\int_{0}^{\frac{\pi}{n}}b(x)\;\!\d x.
\end{equation}
Note that these formulas are valid if $a$ has a fast enough decay, which will be the case in the sequel.
In addition, note also that the last term depends only on the $0$-th Fourier coefficient of the function $b$.
Our next aim is thus to show that $[\W_{in,\kappa}^-,\W_{in,\kappa}^{+\#}] = -\I_{\rm p}(H_{in,\kappa})$
can be rewritten in the above form.

Recall now that $\Xi_{\frac{1}{2}}(D)$ is a unitary operator, with $\Xi_{\frac{1}{2}}(D)^*=\Xi_{\frac{1}{2}}(-D)$.
One also infers from the definition in \eqref{def_Xi} that $\Xi_{in}(D)$ is invertible with $\Xi_{in}(D)^{-1} = \Xi_{in}(-D)$
and that $\Xi_{in}(D)^* = \Xi_{-in}(-D)$.
Then one gets
\begin{align*}
& \Xi_{in}(D)^{-1} \Xi_{\frac{1}{2}}(D)\Big(\W_{in,\kappa}^- \W_{in,\kappa}^{+\#}
- \W_{in,\kappa}^{+\#} \W_{in,\kappa}^-\Big)\Xi_{\frac{1}{2}}(D)^* \Xi_{in}(D)\\
& = \big(\one-\varsigma G^+_n(D)\e^{2inX}\big)F_{in,\kappa}(X)\big(\varsigma \e^{2inX}G^-_n(D)-\one\big)-\one,
\end{align*}
where
\begin{equation*}
F_{in,\kappa}(X):=\frac{-1}{(1 -\varsigma \e^{- \pi n}\e^{2inX})(1-\varsigma \e^{\pi n}\e^{2inX})}
\end{equation*}
and
\begin{equation*}
G_n^{\pm}(\xi):=\Xi_{\pm in}(-\xi)\Xi_{\mp in}(\xi).
\end{equation*}
From the identity $\Gamma(z+\frac{1}{2})\Gamma(-z+\frac{1}{2})=\frac{\pi}{\cos(\pi z)}$ one then infers that
$$
G_n^{\pm}(\xi)=\e^{\pm \pi n}\frac{\e^{\pi \xi}+\e^{\mp \pi n}}{\e^{\pi \xi}+\e^{\pm \pi n}},
$$
and by taking into account the identity \cite[(6.10)]{DR} written in our framework, namely
\begin{equation*}
\e^{2inX}G_n^-(D)+G_n^+(D)\e^{2inX}=2\cosh(\pi n)\e^{2inX},
\end{equation*}
one then gets
\begin{align*}
& \big(\one-\varsigma G^+_n(D)\e^{2inX}\big)F_{in,\kappa}(X)\big(\varsigma \e^{2inX}G^-_n(D)-\one\big)-\one \\
& = -  F_{in,\kappa}(X) - \varsigma^2G_n^+(D)\e^{2inX}F_{in,\kappa}(X)\e^{2inX}G_n^-(D) \\
&\quad +\varsigma F_{in,\kappa}(X)\e^{2inX}G_n^-(D) + \varsigma G_n^+(D)\e^{2in X}F_{in,\kappa}(X) -\one\\
&=- F_{in,\kappa}(X)-\varsigma^2\e^{2inX}F_{in,\kappa}(X)\e^{2inX}-\varsigma^2G_n^+(D)\big[\e^{2inX}F_{in,\kappa}(X)\e^{2inX},G_n^-(D)\big] \\
&\quad + \varsigma F_{in,\kappa}(X)\big(2\cosh(\pi n) \e^{2inX} -G_n^+(D)\e^{2in X}\big)+ \varsigma G_n^+(D)\e^{2in X}F_{in,\kappa}(X) -\one\\
&=-F_{in,\kappa}(X)\big(1- 2 \varsigma \cosh (\pi n) \e^{2inX} + \varsigma^2\e^{4inX}\big) -\one\\
&\quad -\varsigma^2G_n^+(D)\big[\e^{2inX}F_{in,\kappa}(X)\e^{2inX},G_n^-(D)\big]  - \varsigma \big[F_{in,\kappa}(X),G_n^+(D)\big]\e^{2in X}\\
&= - \underbrace{ \varsigma^2 G_n^+(D)\big[\e^{2inX}F_{in,\kappa}(X)\e^{2inX},G_n^-(D)\big]}_{=:I_{in,\kappa}}
- \underbrace{\varsigma \big[F_{in,\kappa}(X),G_n^+(D)\big]\e^{2in X}}_{=:J_{in,\kappa}}
\end{align*}

For the last step, observe that since the function $F_{in,\kappa}$ is a smooth $\frac{\pi}{n}$-periodic function,
its Fourier series converges uniformly. We can thus write $F_{in,\kappa}(X)=\sum_{\ell\in\Z}c_\ell \e^{2in\ell X}$.
Using the relation $\e^{isX}g(D)\e^{-isX}=g(D-s)$, which holds for any $g\in C_{\rm b}(\R)$ and $s\in\R$, we obtain
\begin{align*}
I_{in,\kappa}&=\varsigma^2\sum_{\ell\in\Z}c_\ell G_n^+(D)\Bigl\{G_n^-\bigl(D-2n(\ell+2)\bigr)-G_n^-(D)\Bigr\} \e^{2in(\ell+2)X}, \\
J_{in,\kappa}&=\varsigma\sum_{\ell\in\Z}c_\ell \Bigl\{G_n^+(D-2n\ell)-G_n^+(D)\Bigr\}  \e^{2in(\ell+1)X}.
\end{align*}
By applying then formula \eqref{formula of tracen} one infers that
$\Trace_n(I_{in,\kappa})=0$ since the $0$-th Fourier coefficient of the corresponding function $b$ is obtained for $\ell=-2$, but the first factor vanishes
precisely when $\ell=-2$.
On the other hand one has
\begin{align*}
\Trace_n(J_{in,\kappa})&=\varsigma c_{-1}\frac{1}{2n}\int_0^{2n}\sum_{k\in\Z}
\Bigl\{G_n^+\bigl(\theta+2n(k+1)\bigr)-G_n^+\bigl(\theta+2nk)\bigr)\Bigr\}\;\!\d\theta\\
&=\varsigma c_{-1}\Bigl\{G_n^+(\infty)-G_n^{+}(-\infty)\Bigr\}\\
&=\varsigma c_{-1}(\e^{\pi n}-\e^{-\pi n}).
\end{align*}
Finally, by collecting the result obtained so far and by using the cyclicity of the traces one gets
\begin{align*}
& \Trace_n\big(\W_{in,\kappa}^- \W_{in,\kappa}^{+\#}
- \W_{in,\kappa}^{+\#} \W_{in,\kappa}^-\big) \\
& = \Trace_n\Big(\Xi_{in}(D)^{-1} \Xi_{\frac{1}{2}}(D)\Big(\W_{in,\kappa}^- \W_{in,\kappa}^{+\#}
- \W_{in,\kappa}^{+\#} \W_{in,\kappa}^-\Big)\Xi_{\frac{1}{2}}(D)^* \Xi_{in}(D)\Big) \\
& =  - \varsigma c_{-1}(\e^{\pi n}-\e^{-\pi n}).
\end{align*}

For the computation of $c_{-1}$ it is enough to observe that if $\frac{1}{n}\ln(|\kappa|)\in (-\pi,\pi)$ then
\begin{align*}
F_{in,\kappa}(x) & = \frac{-1}{(1 -\varsigma \e^{- \pi n}\e^{2inx})(1-\varsigma \e^{\pi n}\e^{2inx})} \\
& = \frac{1}{\varsigma\e^{\pi n}} \e^{-2inx}\;\! \frac{1}{1-\varsigma \e^{-\pi n}\e^{2inx}}\ \frac{1}{1-\varsigma^{-1} \e^{-\pi n}\e^{-2inx}} \\
& = \frac{1}{\varsigma\e^{\pi n}} \e^{-2inx}\;\!  \sum_{j=0}^\infty \big(\varsigma \e^{-\pi n}\e^{2inx}\big)^j\;\!
\sum_{k=0}^\infty \big(\varsigma^{-1} \e^{-\pi n}\e^{-2inx}\big)^k,
\end{align*}
from which one infers by considering the diagonal sum that
$$
c_{-1} = \varsigma^{-1}\e^{-\pi n} \sum_{j=0}^\infty \big(\e^{-\pi n}\big)^{2j} = \varsigma^{-1} \frac{\e^{-\pi n}}{1-\e^{-2\pi n}}
=\varsigma^{-1}\frac{1}{\e^{\pi n}-\e^{-\pi n}}\ .
$$
It follows that $-\Trace_n(\I_{\rm p}\big(H_{in,\kappa})\big)=-1$ if
$\frac{1}{n}\ln(|\kappa|)\in (-\pi,\pi)$.
On the other hand, if $\frac{1}{n}\ln(|\kappa|)>\pi$ or if $\frac{1}{n}\ln(|\kappa|)<-\pi$ then one gets by a similar argument that the function $F_{in,\kappa}$ has a Fourier series with coefficient $c_{-1}$ equal to $0$. In such a case one gets $-\Trace_n(\I_{\rm p}\big(H_{in,\kappa})\big)=0$, as expected.
\end{proof}

\begin{remark}
Let us emphasize that the previous theorem is the first topological version of Levinson's theorem when an infinite number of eigenvalues is involved.
Note however that a \emph{generalized Levinson's theorem} involving an infinite
number of bound states already appeared in \cite{Rai, Tie}, but it corresponds to a
relation between the asymptotic behaviors of the spectral shift function and of the
eigenvalues counting functions.
A deeper understanding of the relation between our result and the results contained in these papers would certainly be valuable.
\end{remark}


\begin{thebibliography}{9}

\bibitem{Ati}
M.F. Atiyah, \emph{Elliptic operators, discrete groups and von Neumann algebras}, Colloque ``Analyse et Topologie'' en l'Honneur de Henri Cartan
(Orsay, 1974), pp. 43--72, Ast\'erisque {\bf 32-33}, Soc. Math. France, Paris, 1976.

\bibitem{BDG}
L. Bruneau, J.  Derezi\' nski, V. Georgescu,
\emph{Homogeneous Schr\"odinger operators on half-line},
Ann. Henri Poincar\'e {\bf 12} no. 3, 547--590, 2011.

\bibitem{DR}
J. Derezi\'nski, S. Richard,
\emph{On Schr\"odinger operators with inverse square potentials on the half line},
Ann. Henri Poincar\'e {\bf 18} no.~3, 869--928, 2017.

\bibitem{FF}
J. Faupin, J. Fr\"olich,
\emph{Asymptotic completeness in dissipative scattering theory},
arXiv:1703.09018.

\bibitem{IR} H. Inoue, S. Richard,
\emph{Index theorems for Fredholm, semi-Fredholm, and almost periodic operators: all in one example}, arXiv:1711.07113.

\bibitem{Lev}
N. Levinson, \emph{On the uniqueness of the potential in a Schr\"odinger equation for a given asymptotic phase},
Danske Vid. Selsk, Mat.-Fys. Medd. {\bf 25} no.~9, 29 pp., 1949.

\bibitem{M}
A. Mostafazadeh, \emph{Physics of spectral singularities},
in the proceedings of the XXXIII workshop on Geometric Methods in Physics,
held in Bialowieza, Poland, June 29 - July 05, 2014; Geometric Methods in Physics, Trends in Mathematics 145--165, Birkh\"auser, 2015.

\bibitem{NPR}
F. Nicoleau, D. Parra, S. Richard, \emph{Does Levinson's theorem count complex eigenvalues?},
J. Math. Phys. {\bf 58}, 102101, 7pp., 2017.

\bibitem{Rai}
G. Raikov,
\emph{Low energy asymptotics of the spectral shift function for Pauli operators with nonconstant magnetic fields},  Publ. Res. Inst. Math. Sci. {\bf 46} no.~3, 565--590, 2010.

\bibitem{RS}
M. Reed, B. Simon,
\emph{Methods of modern mathematical physics IV: Analysis of operators},
Academic Press, Inc., 1978.

\bibitem{R}
S. Richard, \emph{Levinson's theorem: an index theorem in scattering theory}, in Spectral theory and mathematical physics,
Operator theory advances and Application Vol. 245, pp. 149--203, Birkh\"auser/Springer, 2016.

\bibitem{S}
B.F. Samsonov, \emph{Hermitian Hamiltonian equivalent to a given non-Hermitian one: manifestation of spectral singularity},
Phil. Trans. R. Soc. A {\bf 371}, 20120044, 2013.

\bibitem{Sch}
J. Schwartz, \emph{Some non-selfadjoint operators},
Comm. Pure Appl. Math. {\bf 13}, 609--639, 1960.

\bibitem{Tie}
R. Tiedra de Aldecoa,
\emph{Asymptotics near $\pm m$ of the spectral shift function for Dirac operators with non-constant magnetic fields},
Comm. Partial Differential Equations {\bf 36} no.~1, 10--41, 2011.

\end{thebibliography}
\end{document}